\documentclass[onecolumn]{article}
\usepackage{amssymb}
\usepackage{amsmath}
\usepackage{amsfonts}
\usepackage{sw20bams}

\newtheorem{theorem}{Theorem}
\newtheorem{acknowledgement}{Acknowledgement}

\newtheorem{corollary}{Corollary}

\newtheorem{definition}{Definition}
\newtheorem{example}{Example}

\newtheorem{notation}{Notation}

\newtheorem{proposition}{Proposition}

\input{tcilatex}
\numberwithin{theorem}{section}
\numberwithin{example}{section}
\numberwithin{remark}{section}
\numberwithin{proposition}{section}
\numberwithin{lemma}{section}
\numberwithin{corollary}{section}
\numberwithin{definition}{section}

\begin{document}

\title{On the Notion of a Generalized Mapping on Multiset Spaces}
\author{Athar Kharal\thanks{%
atharkharal@gmail.com} \\
Department of Mathematics, Air University, Pakistan \and Mansoor H. Alshehri%
\thanks{%
Corresponding Author, mhalshehri@ksu.edu.sa}, Nasser Bin Turki\thanks{%
nassert@ksu.edu.sa}, Faisal Z. Duraihem\thanks{%
faldureiham@ksu.edu.sa} \\
Department of Mathematics, College of Science, \\
King Saud University, P.O. Box 2455, Riyadh 11451, Saudi Arabia}
\date{ }
\maketitle

\begin{abstract}
A sufficiently generalized concept of mappings on multisets has been
introduced, thus resolving a long standing obstacle in structural study of
multiset processing. It has been shown that the mapping defined herein can
model a vast array of already defined mappings within the domain of
Theoretical Computer Science as special cases and also it handels diverse
situations in multiset rewriting transformations. Specifically, this paper
unifies and generalizes the works of Parikh \cite{parikh66} (1966), Hickman 
\cite{hickman80} (1980), Khomenko \cite{Khomenko03} (2003) and Nazmul \cite%
{13namul} (2013).\bigskip 

\textbf{Kewords:} Multisets mapping; Multiset process modeling; Parikh
mapping; Multiset rewriting, Hickman mapping
\end{abstract}

\section{\protect\bigskip \textbf{Introduction}}

Index object in Pandas, the Data Science library of Python language, is a
multiset. Multisets are also important for analysis of various computer
algorithms \cite{dramn20a,dramn20b,luo20corr,nam20hook}. Although in
mathematics world, one can trace the origins of modern interest in multisets
as far as the great mathematician Richard Dedekind (d. 1916), but it is only
with the recent advent of Computer Science and allied fields that a renewed
interest is being witnessed \cite{feli20genMST}. Today in search of ever
more generalization, multisets have become the dominant data structures of
contemporary Computer Science\cite{wang19stanley}, Cybernetics \cite%
{burgin92}, Information Science, Combinatorics, and unconventional
computational paradigms like membrane \cite{paun06} and DNA computing, and
Petri Nets \cite{Khomenko03}. A set is a well defined collection of distinct
objects. As weakening the condition of well definedness gave birth to Fuzzy
Sets, weakening of the condition of `distinctness' produces the notion of
multisets. Thus multisets are sets in which repetition of elements is
significant. Both fuzzy sets and multisets are generalizations of classical
sets. Examples of multisets abound: water molecule $H_{2}O$ is a multiset as 
$H_{2}O=\left\{ H,H,O\right\} =\left\{ 2/H,1/O\right\} ,$ the prime
factorization of integers $n>0$ is a multiset whose elements are primes.
Every monic polynomial $f(x)$ over the complex numbers corresponds in a
natural way to the multiset of its roots. Zeros and poles of meromorphic
functions, invariants of matrices in a canonical form are multisets. Again
words in a language are multisets on the set of alphabets $\sum ,$ symbols
in a membrane are multisets over the set of alphabets and markings of a
Petri net are multisets over the set of places. Thus naturally multisets
find numerous applications in diverse fields: Database Management Systems,
Cryptography, Membrane Computing, Rewriting Systems, Abstract Chemical
Machines and Neural Networks etc \cite{gil19grf,satya20soft,zhang20corr}.

Despite a historically prolonged presence of multisets in Mathematics and
Computer Science, research on the multiset theory has not yet gained ground
and is still in its infant stages. The research shows a strong analogy in
the behavior of multisets and sets. Singh \cite{singh07} rightly pointed out
that so far all mappings defined for multisets turn into mappings between
the root sets and thus become ordinary mappings of insignificant
consequences. Since the case in consideration has a wider generality
(compared to that of sets), the results obtained for multisets are
technically more complicated and should be more general ones. Main obstacle,
in a full fledged research, has been the non-availability of a sufficiently
generalized notion of mapping between arbitrary collections of multisets.

This paper addresses the problem of extending a mapping defined between the
root sets of two multiset spaces to the multiset spaces themselves.
Organization of the paper is as follows: Section \ref{sec_RltdWrk} surveys
the related literature and locates the present work in its context. Section %
\ref{sec_Prelim} collects necessary definitions and presents some new
requisite results. Section \ref{sec_maps} is first of the two main sections
of this work. It formally introduces the new notion of mapping on multiset
spaces. This section also presents many important results, properties and
insights concerning the new mapping. Section \ref{sec_Interlinks} is the
other main section which demonstrates the relevance and interconnections of
the notion presented herein with other established notions of the field e.g.
Parikh \cite{parikh66}, Hickman \cite{hickman80}, Khomenko \cite{Khomenko03}
and Nazmul \cite{13namul} mappings also with cardinality, similarity and
distance measures on multisets. It shows how the mapping presented herein
unifies and generalizes various notions. Section \ref{sec_finalRems}
concludes the paper.

\section{\protect\bigskip \textbf{Related Work}\label{sec_RltdWrk}}

Parikh \cite{parikh66} introduced a mapping which interlinked words in a
language with arithmetical vectors. Despite the fact that Parikh mappings
find many applications \cite{amir03}, this notion has been found deficient
in many aspects by the subsequent researchers \cite{Karhumaki80,Mateescu01}.

Hickman also introduced a notion of mapping on multisets in \cite{hickman80}%
. His definition (see Definition \ref{def_hickman} later in this paper) is
somewhat restrictive definition of a multiset mapping, for multisets are
supposed to be richer objects in some sense than ordinary sets, and one
might expect that this extra richness would be reflected in the definition
of a multiset mapping, but the definition assumes identical multiplicity
limits.

Manjunath and John \cite{manjunath06} have done some preliminary work on
multiset relations. They defined multiset as a sub multiset of the
generalized Cartesian product. Continuing with same line of thought Girish
and John introduced a notion of functions in multiset context in \cite%
{mst09gir}. Girish and John defined functions as a subcollection of
Cartesian product of two multisets. In this specific context i.e. function
as subset of multiset relation, they have obtained results involving
different types of functions. It must be noted that their approach does not
consider a function defined on the root sets and its extension to the
multiset spaces induced by the root sets.

Singh \cite{singh07} considers ordinary mappings between the roots sets and
multiplicities are left unaltered. This thus makes the results for multisets
as exact copies of the classical results of set theory. Whereas, considering
multisets a generalization of sets, one naturally expects to see divergences.

Recently Nazmul \textit{et al}. have introduced yet another notion of
mapping on multisets in \cite{13namul}. This definition assumes that
multiplicities of domain and range multiset spaces are same, in fact
identical.

Our work encompasses the works of Parikh, Hickman and Nazmul as special
cases, while at the same time, being distinct from the approach adopted by
Girish and John in \cite{mst09gir}. We consider mappings between root sets 
\textit{extended} to arbitray multiset spaces.

\section{\textbf{Preliminaries\label{sec_Prelim}}\protect\bigskip}

In this section requisite definitions, notations and some other results have
been collected. Interested reader may find further material in \cite%
{mst09gir,mst01jen} and in the references therein.

A well defined collection of elements containing duplicates is called a
multiset. Formally, if $X$ is a set of elements, a multiset $M$ drawn from
the set $X$ is represented by a count function $C_{M}$ defined as $%
C_{M}:X\rightarrow N$ where $N$ represents the set of non negative integers.
For each $x\in X,C_{M}(x)$ is the characteristic value of $x$ in $M$ and
indicates the number of occurrences of the element $x$ in $M$. A multiset $M$
is a set if $C_{M}(x)\leq 0$ or $1$ $\forall $ $x\in X.$ The multiset space $%
X^{m}$ is the set of all multisets whose elements are in $X$ such that no
element in the multiset occurs more than $m$ times; more formally, if $%
X=\{x_{1},x_{2},...,x_{k}\}$ then $X^{m}=\{%
\{m_{1}/x_{1},m_{2}/x_{2},...,m_{k}/x_{k}\}$~$|$~$i=1,2,...,k$; $m_{i}\in
\{0,1,2,...,m\}\}$. The set $X^{\infty }$ is the set of all multisets over a
domain $X$ such that there is no limit to the number of occurrences of an
element in a multiset.

\begin{notation}
$\left( 1\right) $A multiset may safely be identified as the same as its
count function. In the context of this paper, a multiset $A\in X^{m}$ and
its count function will be denoted by the same letter $A.$ Intended use
would be clear by context i.e. by $A$ we shall mean multiset and by $A\left(
x\right) $ we shall mean the count of $x$ in multiset $A.$

$\left( 2\right) $ In the sequel $m^{\ast }$ will denote the set $\left\{
0,1,2,...,m\right\} ,$ for any arbitrary $m\in \mathbb{N}.$
\end{notation}

Let $M$ and $N$ be two multisets drawn from a multiset space $X^{m}$. $M$ is
a sub multiset of $N~(M\subseteq N)$ if $M(x)\leq N(x)$ for all $x\in X$. $M$
is a proper sub multiset of $N$ $(M\subset N)$ if $M(x)\leq N(x)$ $\forall $ 
$x\in X$ and there exists at least one $x\in X$ such that $M(x)<N(x)$. The
union (respectively, intersection, difference) of two multisets $M$ and $N$
drawn from a set $X$ is a multiset $P$ denoted by $M\cup N$ (resp., $M\cap
N,~M-N$) such that $\forall ~x\in X,$ $\left( M\cup N\right) (x)=\max
\{M(x),N(x)\}~\left( \left( M\cap N\right) (x)=\min \{M(x),N(x)\},~~\left(
M-N\right) \left( x\right) =\max \left\{ 0,M\left( x\right) -N\left(
x\right) \right\} \right) $. The complement $M^{c}$ of $M$ in $X^{m}$ is
given as $M^{c}\left( x\right) =m-M\left( x\right) ~\forall x\in X$.

We now define the notion of constant multiset and observe that the notion of
empty multiset, as defined by Jena \textit{et.al.} (Definition 0.1(v) \cite%
{mst01jen}), is a special kind of constant multiset.

\begin{definition}
A constant multiset $\widetilde{t}$ in $X^{m}$ is defined as $\widetilde{t}%
\left( x\right) =t~\forall x\in X$ where $t$ is some integer such that $%
0\leq t\leq m.$ Following special constant multisets seem to be interesting:

$\left( i\right) $ Multiset $\widetilde{0}$, defined as $\widetilde{0}\left(
x\right) =0~\forall x\in X.$ This is empty multiset \cite{mst01jen}.

$\left( ii\right) $ Multiset $\widetilde{1}$, defined as $\widetilde{1}%
\left( x\right) =1~\forall x\in X.$ Clearly $\widetilde{1}$ is the ordinary
set $X$.

$\left( iii\right) $ Multiset $\widetilde{m}$, defined as $\widetilde{m}%
=m~\forall x\in X.$ This is called absolute multiset in $X^{m}.$
\end{definition}

For a few basic properties of multiset union and intersection, we refer to
Theorem 1.1 \cite{mst01jen}. Moreover, we have some more properties:

\begin{theorem}
Let $M,N\in X^{m}$ then, we have:%
\begin{equation*}
\begin{tabular}{lll}
$\left( 1\right) ~~~M\cup \widetilde{0}=M.$ & ~~~~~ & $\left( 11\right) ~~$ $%
M\cap \widetilde{0}=\widetilde{0}.$ \\ 
$\left( 2\right) $~$~~M\cup \widetilde{m}=\widetilde{m}.$ & ~~~~~ & $\left(
12\right) $~$~~M\cap \widetilde{m}=M.$ \\ 
$\left( 3\right) $~$~~M\cup M=M.$ & ~~~~~ & $\left( 13\right) $~$~~M\cap
M=M. $ \\ 
$\left( 4\right) $~$~~M\cup \left( N\cup P\right) =\left( M\cup N\right)
\cup P.$ & ~~~~~ & $\left( 14\right) $~$~~M\cap \left( N\cap P\right)
=\left( M\cap N\right) \cap P.$ \\ 
$\left( 5\right) $~$~~M\subseteq N\iff M\cup N=N.$ & ~~~~~ & $\left(
15\right) $~$~~M\subseteq N\iff M\cap N=M.$ \\ 
$\left( 6\right) ~~~\widetilde{0}^{c}=\widetilde{m}.$ &  & $\left( 16\right)
~~~\widetilde{m}^{c}=\widetilde{0}.$ \\ 
$\left( 7\right) ~~~M\subseteq N\iff N^{c}\subseteq M^{c}.$ &  & $\left(
17\right) ~~~M-N=M-\left( M\cap N\right) .$ \\ 
$\left( 8\right) ~~~M-N=\left( M\cup N\right) -N.$ &  & $\left( 18\right)
~~~\left( M-N\right) -P\subseteq M-\left( N\cup P\right) .$ \\ 
$\left( 9\right) ~~~\left( M-N\right) -P\subseteq \left( M-N\right) \cup
\left( M\cap P\right) .$ &  & $\left( 19\right) ~~~\left( M-P\right) \cup
\left( N-P\right) \subseteq \left( M\cup N\right) -P.$ \\ 
$\left( 10\right) ~~~M-\left( N\cup P\right) \subseteq \left( M-N\right)
\cap \left( M-P\right) .$ &  & 
\end{tabular}%
\end{equation*}
\end{theorem}

\begin{proof}
We only prove $\left( 10\right) ,\left( 17\right) $ and $\left( 18\right) ,$
remaining proofs are similar.

$\left( 10\right) ~\left( M-\left( N\cup P\right) \right) \left( x\right)
=\max \left\{ 0,M\left( x\right) -\left( N\cup P\right) \left( x\right)
\right\} =\max \left\{ 0,M\left( x\right) -\max \left( N\left( x\right)
,P\left( x\right) \right) \right\} \leq \min \left\{ \left( M\left( x\right)
-N\left( x\right) \right) ,\left( M\left( x\right) -P\left( x\right) \right)
\right\} =\allowbreak \left( \left( M-N\right) \cap \left( M-P\right)
\right) \left( x\right) .$

$\left( 17\right) ~\left( M-\left( M\cap N\right) \right) \left( x\right)
=\max \left\{ 0,M\left( x\right) -\left( M\cap N\right) \left( x\right)
\right\} =\max \left\{ 0,M\left( x\right) -\min \left( M\left( x\right)
,N\left( x\right) \right) \right\} =\max \left\{ 0,M\left( x\right) -N\left(
x\right) \right\} =\left( M-N\right) \left( x\right) .$

$\left( 18\right) ~\left( \left( M-N\right) -P\right) \left( x\right) =\max
\left\{ 0,\left( M-N\right) \left( x\right) -P\left( x\right) \right\} =\max
\left\{ 0,\max \left\{ 0,M\left( x\right) -N\left( x\right) \right\}
-P\left( x\right) \right\} \leq \max \left\{ 0,M\left( x\right) -\max \left(
N\left( x\right) ,P\left( x\right) \right) \right\} =\allowbreak \max
\left\{ 0,M\left( x\right) -\left( N\cup P\right) \left( x\right) \right\}
=\left( M-\left( N\cup P\right) \right) \left( x\right) .$
\end{proof}

Multi Set Theory (MST) is an extension of Classical Set Theory (CST). Thus
one naturally expects divergences between the new theory and its classical
counterpart. One such case in point, is the Law of Excluded Middle and Law
of Non-Contradiction. It is notable that the set theoretic forms of these
laws do not hold i.e. for $A\in X^{m},$ in general, we have: $A\cap
A^{c}\not=\widetilde{0}$ and $A\cup A^{c}\not=\widetilde{m}.$
Counterexamples supporting above statements may be seen by choosing $%
A=\left\{ 4/a,2/b,0/c,0/d,3/e\right\} \in \left\{ a,b,c,d,e\right\} ^{4}$and
calculating%
\begin{eqnarray*}
A\cap A^{c} &=&\left\{ 0/a,1/b,0/c,0/d,2/e\right\} \neq \widetilde{0}, \\
A\cup A^{c} &=&\left\{ 4/a,3/b,4/c,4/d,2/e\right\} \neq \widetilde{m},
\end{eqnarray*}

In classical sets we have two sets $A$ and $B$ disjoint, symbolically, $%
A\cap B=\phi $ if and only if $A\subseteq B^{c}.$ Since disjointedness does
not make sense in multisets, naturally one looks for some alternate
condition which can guarantee the subsethood in complement of the bigger
multiset. This thought is formalized below in the notion of $m$-coincidence
as:

\begin{definition}
Two multisets $A,B\in X^{m}$ are said to be $m$-coincident if we have $%
A\left( x\right) +B\left( x\right) >m$ for some $x\in X.$ We denote this as $%
AmB.$ If $A$ is not $m$-coincident then we write $A\NEG{m}B.$
\end{definition}

Some immediate consequences of the above definition are:

\begin{enumerate}
\item $A\NEG{m}B\iff A\subset B^{c}$

\item $A\subset B\Rightarrow A\NEG{m}B^{c}.$
\end{enumerate}

An interesting result regarding $m$-coincidence is given as Theorem \ref%
{th_AlgOfFns} $\left( 15\right) .$ For a detailed study of the same notion
we refer to \cite{KharalMTop}.

\section{\textbf{\protect\bigskip The Mappings}\label{sec_maps}}

In this section we intend to define a suitable notion of mapping on
multisets. By suitable notion we mean that the mapping should be generalized
enough to be defined between arbitrary choices of $X,Y,m,n$ and thus
arbitrary multiset spaces $X^{m}$ and $Y^{n}.$ Moreover a multiset carries
some resemblance with a fuzzy set, in the sense that both notions assign a
numerical value to the elements of some arbitrary set. That is why one
naturally expects the mapping on multiset to be somewhat similar to standard
mapping on fuzzy sets as defined by Zadeh in 1965. Difference between fuzzy
and multisets is that of positive integral values in a multisets. So for a
truly generalized notion of multiset mapping, one has to ensure order
preservation between the special sets $m^{\ast }$ and $n^{\ast }$ related to
multiset spaces $X^{m},$ $Y^{n},$ respectively. In fact, such a truly
generalized mapping is needed to model, for example, multiset rewriting
systems. We shall further discuss this point in subsequent discussion.

For this we first introduce the notion of order preserving maps and then use
these for defining multiset mappings. Recall that $m^{\ast }=\left\{
0,1,2,...,m\right\} .$ We shall call a mapping $p:m^{\ast }\rightarrow
n^{\ast }$ order preserving (briefly, OP), if it satisfies:\newline
$\left( op1\right) ~~~p\left( 0\right) =0$\newline
$\left( op2\right) ~~~p\left( m\right) =n$\newline
$\left( op3\right) ~~~p\left( i\right) \geq p\left( i-1\right) $

An OP map does not allow a crossing of arrows in a traditional mapping
diagram as shown in figure below:%
\begin{equation*}
\FRAME{itbpF}{3.3148in}{2.5495in}{0in}{}{}{opmapfig.gif}{\special{language
"Scientific Word";type "GRAPHIC";maintain-aspect-ratio TRUE;display
"USEDEF";valid_file "F";width 3.3148in;height 2.5495in;depth
0in;original-width 4.9796in;original-height 3.8233in;cropleft "0";croptop
"1";cropright "1";cropbottom "0";filename 'OpMapfig.GIF';file-properties
"XNPEU";}}
\end{equation*}%
Following are immediate observations from the definition of OP maps:

\begin{proposition}
\label{pr_OPrslts}If $p:m^{\ast }\rightarrow n^{\ast }$ is an OP map then

$\left( 1\right) $ $p$ is a constant map iff $n=0.$

$\left( 2\right) $ $p$ may be a surjective map, only if $n\leq m.$

$\left( 3\right) $ $p$ may be an injective map only if $n\geq m.$

$\left( 4\right) $ $p$ may be an bijective map only if $m=n.$

$\left( 5\right) $ If $m=n$ and $p$ is surjective, then $p$ is bijective.

$\left( 6\right) $ If $m=n$ and $p$ is injective, then $p$ is bijective.

$\left( 7\right) $ $p$ is bijective iff $p$ is identity map.

$\left( 8\right) ~p\left( \min \left( x,y\right) \right) =\min \left(
p\left( x\right) ,p\left( y\right) \right) $

$\left( 9\right) ~p\left( \max \left( x,y\right) \right) =\max \left(
p\left( x\right) ,p\left( y\right) \right) .$
\end{proposition}

\begin{definition}
Let $X^{m},Y^{n}$ be multiset spaces with $m,n\in \mathbb{N}$ and $X,Y$,
arbitrary sets. Then a mapping $f=\left( u,p\right) :X^{m}\rightarrow Y^{n}$
is said to be multiset mapping, if the component mapping $u:X\rightarrow Y$
is an ordinary map and $p:m^{\ast }\rightarrow n^{\ast }$ is an OP map.
Furthermore a multiset map $f$ is said to be $u$-injective (resp., $u$%
-surjective, $u$-bijective, $p$-injective, $p$-surjective, $p$-bijective) if 
$u$ (resp., $p$) is an injective (resp., surjective, bijective) map. $f$ is
said to be injective (resp., surjective, bijective) if it is both $u$%
-injective (resp., $u$-surjective, $u$-bijective) and $p$-injective (resp., $%
p$-surjective, $p$-bijective).
\end{definition}

\begin{definition}
For a multiset mapping $f=\left( u,p\right) :X^{m}\rightarrow Y^{n}$, the
image and preimage of multisets $A\in X^{m},$ $M\in Y^{n}$ are given,
respectively, as 
\begin{equation*}
f\left( A\right) \left( y\right) =\left\{ 
\begin{array}{ccc}
p\left( \underset{x\in u^{-1}\left( y\right) }{\dbigvee }A\left( x\right)
\right) & if & u^{-1}\left( y\right) \neq \phi \\ 
&  &  \\ 
0 &  & otherwise%
\end{array}%
\right.
\end{equation*}%
\begin{equation*}
f^{-1}\left( M\right) \left( x\right) =\left\{ 
\begin{array}{ccc}
\dbigvee p^{-1}\left( M\left( u\left( x\right) \right) \right) & if & 
p^{-1}\left( M\left( u\left( x\right) \right) \right) \neq \phi \\ 
&  &  \\ 
0 &  & otherwise%
\end{array}%
\right.
\end{equation*}
\end{definition}

Our notion is general one and does not restrict the size of $Y$ and $n,$
thus a totally arbitrary choice of multiset spaces may be made to model a
vast number of situations. Following example illustrates the necessary
calculations involved in implementation of above definition.

\begin{example}
Let $X=\left\{ a,b,c,d\right\} $ and $Y=\left\{ s,t,x,y,z\right\} .$
Consider $f=\left( u,p\right) :X^{4}\rightarrow Y^{5}$ with%
\begin{equation*}
~~u\left( a\right) =y,u\left( b\right) =y,u\left( c\right) =z,u\left(
d\right) =s,~~~p\left( 0\right) =0,p\left( 1\right) =1,p\left( 2\right)
=5,p\left( 3\right) =5,p\left( 4\right) =5
\end{equation*}%
Choose%
\begin{equation*}
A=\left\{ 1/a,4/b,2/c,4/d\right\} \in X^{4},~~~M=\left\{
1/s,2/t,1/x,1/y,5/z\right\} \in Y^{5},
\end{equation*}%
Then calculations show%
\begin{eqnarray*}
f\left( A\right) \left( s\right) &=&p\left( \underset{x^{\prime }\in
u^{-1}\left( s\right) }{\dbigvee }A\left( x^{\prime }\right) \right)
=p\left( \dbigvee A\left( d\right) \right) =p\left( \dbigvee \left( 4\right)
\right) =p\left( 4\right) =5 \\
f\left( A\right) \left( t\right) &=&0\text{ as }u^{-1}\left( t\right) =\phi
\\
f\left( A\right) \left( x\right) &=&0\text{ as }u^{-1}\left( x\right) =\phi
\\
f\left( A\right) \left( y\right) &=&p\left( \underset{x^{\prime }\in
u^{-1}\left( y\right) }{\dbigvee }A\left( x^{\prime }\right) \right)
=p\left( \underset{x^{\prime }\in \left\{ a,b\right\} }{\dbigvee }A\left(
x^{\prime }\right) \right) =p\left( \dbigvee \left( A\left( a\right)
,A\left( b\right) \right) \right) =p\left( \dbigvee \left( 1,4\right)
\right) =p\left( 4\right) =5 \\
f\left( A\right) \left( z\right) &=&p\left( \underset{x^{\prime }\in
u^{-1}\left( z\right) }{\dbigvee }A\left( x^{\prime }\right) \right)
=p\left( \dbigvee A\left( c\right) \right) =p\left( \dbigvee \left( 2\right)
\right) =p\left( 2\right) =5
\end{eqnarray*}%
Hence $f\left( A\right) =\left\{ 5/s,0/t,0/x,5/y,5/z\right\} .$ Similarly
for the preimage of $M$ we have 
\begin{eqnarray*}
f^{-1}\left( M\right) \left( a\right) &=&\dbigvee p^{-1}\left( M\left(
u\left( a\right) \right) \right) =\dbigvee p^{-1}\left( M\left( y\right)
\right) =\dbigvee p^{-1}\left( 1\right) =1 \\
f^{-1}\left( M\right) \left( b\right) &=&\dbigvee p^{-1}\left( M\left(
u\left( b\right) \right) \right) =\dbigvee p^{-1}\left( M\left( y\right)
\right) =\dbigvee p^{-1}\left( 1\right) =1 \\
f^{-1}\left( M\right) \left( c\right) &=&\dbigvee p^{-1}\left( M\left(
u\left( c\right) \right) \right) =\dbigvee p^{-1}\left( M\left( z\right)
\right) =\dbigvee p^{-1}\left( 5\right) =\dbigvee \left( 2,3,4\right) =4 \\
f^{-1}\left( M\right) \left( d\right) &=&\dbigvee p^{-1}\left( M\left(
u\left( d\right) \right) \right) =\dbigvee p^{-1}\left( M\left( s\right)
\right) =\dbigvee p^{-1}\left( 1\right) =1 \\
&\Rightarrow &f^{-1}\left( M\right) =\left\{ 1/a,1/b,4/c,1/d\right\} .
\end{eqnarray*}
\end{example}

\begin{theorem}
\label{th_AlgOfFns}For a multimap $f=\left( u,p\right) :X^{m}\rightarrow
Y^{n}$ and $A,B\in X^{m},$ $M,N\in Y^{n},$ we have

$\left( 1\right) $ $f\left( \widetilde{0}\right) =\widetilde{0}$

$\left( 2\right) $ $f\left( \widetilde{m}\right) \subseteq \widetilde{n},$
equality hold if $f$ is $u$-injective

$\left( 3\right) $ $f\left( A\cup B\right) =f\left( A\right) \cup f\left(
B\right) $

$\left( 4\right) $ $f\left( A\cap B\right) \subseteq f\left( A\right) \cap
f\left( B\right) ,~$equality holds if $f$ is $u$-injective

$\left( 5\right) ~A\subseteq B\Rightarrow f\left( A\right) \subseteq f\left(
B\right) $

$\left( 6a\right) $ $\left( f\left( A\right) \right) ^{c}\subseteq f\left(
A^{c}\right) $ if $f$ is $u$-surjective $p$-bijective.

$\left( 6b\right) $ $f\left( A^{c}\right) \subseteq \left( f\left( A\right)
\right) ^{c}$ if $f$ is $u$-injective $p$-bijective. Equality holds if $f$
is bijective.

$\left( 7\right) $ $\widetilde{0}\subseteq f^{-1}\left( \widetilde{0}\right) 
$

$\left( 8\right) $ $\widetilde{m}=f^{-1}\left( \widetilde{n}\right) $

$\left( 9\right) ~f^{-1}\left( M\cup N\right) \subseteq f^{-1}\left(
M\right) \cup f^{-1}\left( N\right) ,$ equality holds if $f$ is $p$%
-surjective

$\left( 10\right) $ $f^{-1}\left( M\right) \cap f^{-1}\left( N\right)
\subseteq f^{-1}\left( M\cap N\right) ,$ equality holds if $f$ is $p$%
-surjective

$\left( 11\right) $ $M\subseteq N\Rightarrow f^{-1}\left( M\right) \subseteq
f^{-1}\left( N\right) $ if $f$ is $p$-surjective

$\left( 12\right) ~f^{-1}\left( M^{c}\right) =\left( f^{-1}\left( M\right)
\right) ^{c}$ if $f$ is $p$-bijective.

$\left( 13\right) ~A\subseteq f^{-1}\left( f\left( A\right) \right) ,$
equality holds if $f$ is $u$-injective $p$-bijective.

$\left( 14\right) ~f\left( f^{-1}\left( M\right) \right) \subseteq M,$
equality holds if $f$ is surjective

$\left( 15\right) $ $f\left( A\right) \NEG{m}f\left( B\right) \Rightarrow A%
\NEG{m}B$ if $f$ is $p$-bijective. Bi-implication holds if $f$ is $u$%
-injective, $p$-bijective.
\end{theorem}

\begin{proof}
We only prove $\left( 4,5\right) $ and $\left( 9\right) ,$ other proofs are
similar.

$\left( 4\right) $ We only consider the non-trivial case when $u^{-1}\left(
y\right) \neq \phi .$ So $f\left( A\cap B\right) \left( y\right)
=\allowbreak p\left( \underset{x\in u^{-1}\left( y\right) }{\dbigvee }\min
\left( A\left( x\right) ,B\left( x\right) \right) \right) \leq p\left( \min
\left( \underset{x\in u^{-1}\left( y\right) }{\dbigvee }\left( A\left(
x\right) ,B\left( x\right) \right) \right) \right) \allowbreak =\min \left(
p\left( \underset{x\in u^{-1}\left( y\right) }{\dbigvee }\left( A\left(
x\right) ,B\left( x\right) \right) \right) \right) =\allowbreak \min \left(
p\left( \underset{x\in u^{-1}\left( y\right) }{\dbigvee }A\left( x\right)
\right) ,p\left( \underset{x\in u^{-1}\left( y\right) }{\dbigvee }B\left(
x\right) \right) \right) \allowbreak =\left( f\left( A\right) \cap f\left(
B\right) \right) \left( y\right) .$ Hence we have $f\left( A\cap B\right)
\subseteq f\left( A\right) \cap f\left( B\right) .$\bigskip

$\left( 5\right) $ $f\left( A\right) =p\left( \underset{x\in u^{-1}\left(
y\right) }{\dbigvee }A\left( x\right) \right) \leq p\left( \underset{x\in
u^{-1}\left( y\right) }{\dbigvee }B\left( x\right) \right) =f\left( B\right)
,$ since $A\left( x\right) \leq B\left( x\right) $ $\forall x\in X.$\bigskip

$\left( 6a\right) $ For each $y\in Y,$ if $f^{-1}\left( y\right) $ is not
empty, then%
\begin{equation*}
f\left( A^{c}\right) \left( y\right) =p\left( \underset{x\in u^{-1}\left(
y\right) }{\dbigvee }A^{c}\left( x\right) \right) =p\left( \underset{x\in
u^{-1}\left( y\right) }{\dbigvee }\left( m-A\left( x\right) \right) \right)
=n-p\left( \underset{x\in u^{-1}\left( y\right) }{\dbigvee }A\left( x\right)
\right)
\end{equation*}%
and%
\begin{eqnarray*}
\left( f\left( A\right) \right) ^{c}\left( y\right) &=&\dbigvee p^{-1}\left(
n-M\left( u\left( x\right) \right) \right) =\dbigvee \left( p^{-1}\left(
n\right) -p^{-1}\left( M\left( u\left( x\right) \right) \right) \right) \\
&=&p^{-1}\left( n\right) -\dbigvee p^{-1}\left( M\left( u\left( x\right)
\right) \right) =m-\dbigvee p^{-1}\left( M\left( u\left( x\right) \right)
\right)
\end{eqnarray*}%
therefore%
\begin{equation*}
f\left( A^{c}\right) \left( y\right) \geq \left( f\left( A\right) \right)
^{c}\left( y\right)
\end{equation*}%
\bigskip

$\left( 9\right) ~f^{-1}\left( M\cup N\right) \left( x\right) =\dbigvee
p^{-1}\left( \max \left( M\left( u\left( .\right) \right) ,N\left( u\left(
.\right) \right) \right) \left( x\right) \right) \allowbreak =\dbigvee
p^{-1}\left( \max \left( M\left( u\left( x\right) \right) ,N\left( u\left(
x\right) \right) \right) \right) \leq \allowbreak p^{-1}\left( \max M\left(
u\left( x\right) \right) \right) \dbigvee p^{-1}\left( \max N\left( u\left(
x\right) \right) \right) =f^{-1}\left( M\right) \left( x\right) \dbigvee
f^{-1}\left( N\right) \left( x\right) \allowbreak =\left( f^{-1}\left(
M\right) \cup f^{-1}\left( N\right) \right) \left( x\right) .$\bigskip

$\left( 11\right) $ 
\begin{equation*}
f^{-1}\left( M\right) \left( x\right) =\dbigvee p^{-1}\left( M\left( u\left(
x\right) \right) \right)
\end{equation*}%
since~$M\subseteq N,~\dbigvee p^{-1}\left( M\left( u\left( x\right) \right)
\right) \leq \dbigvee p^{-1}\left( N\left( u\left( x\right) \right) \right) $
for all $x\in X.$ Hence $f^{-1}\left( M\right) \subseteq f^{-1}\left(
N\right) .$\bigskip

$\left( 13\right) $ $f^{-1}\left( f\left( A\right) \right) \left( y\right)
=\dbigvee p^{-1}\left( f\left( A\right) \left( u\left( x\right) \right)
\right) =\dbigvee p^{-1}\left( p\left( \underset{z\in u^{-1}\left( u\left(
x\right) \right) }{\dbigvee }f^{-1}\left( M\right) \left( z\right) \right)
\right) \geq A\left( x\right) $\bigskip

$\left( 14\right) $ If $u^{-1}\left( y\right) $ is not empty,%
\begin{equation*}
f\left( f^{-1}\left( M\right) \right) \left( y\right) =p\left( \underset{%
x\in u^{-1}\left( y\right) }{\dbigvee }f^{-1}\left( M\right) \left( x\right)
\right) =p\left( \underset{x\in u^{-1}\left( y\right) }{\dbigvee }\left(
\dbigvee p^{-1}\left( M\left( u\left( x\right) \right) \right) \right)
\right) \leq M\left( y\right) .
\end{equation*}
\end{proof}

Suitable counterexamples may be constructed to show the direction of
inclusions.

\begin{example}
We show that $\left( 4\right) ,\left( 6b\right) ,\left( 13\right) $ and $%
\left( 14\right) ,$ are in general irreversible. For $\left( 4\right)
,\left( 13\right) $ and $\left( 14\right) $ consider $f=\left( u,p\right)
:\left\{ a,b,c,d\right\} ^{5}\rightarrow \left\{ x,y,z\right\} ^{7}$ where%
\begin{equation*}
u\left( a\right) =z,u\left( b\right) =z,u\left( c\right) =x,u\left( d\right)
=y,~~~~~~p\left( 0\right) =0,p\left( 1\right) =0,p\left( 2\right) =4,p\left(
3\right) =5,p\left( 4\right) =5,p\left( 5\right) =7
\end{equation*}%
and choose%
\begin{equation*}
A=\left\{ 4/a,0/b,0/c,4/d\right\} ,~~B=\left\{ 1/a,2/b,4/c,4/d\right\}
,~~M=\left\{ 1/x,2/y,6/z\right\}
\end{equation*}%
Then the calculations show%
\begin{eqnarray*}
f\left( A\right) \cap f\left( B\right) &=&\{0/x,5/y,4/z\}\not\subseteq
\{0/x,5/y,0/z\}=f\left( A\cap B\right) \\
f^{-1}\left( f\left( A\right) \right) &=&\{4/a,4/b,1/c,4/d\}\not\subseteq
\{4/a,0/b,0/c,4/d\}=A \\
M &=&\{1/x,2/y,6/z\}\not\subseteq \{0/x,0/y,0/z\}=f\left( f^{-1}\left(
M\right) \right)
\end{eqnarray*}%
Again for $\left( 6b\right) $ set $f=\left( u,p\right) :\left\{
a,b,c,d\right\} ^{7}\rightarrow \left\{ s,t,x,y,z\right\} ^{7}$ to be a $u$%
-injective $p$-bijective map by choosing%
\begin{equation*}
u\left( a\right) =s,u\left( b\right) =z,u\left( c\right) =x,u\left( d\right)
=y,~~~\text{and }p:m^{\ast }\rightarrow n^{\ast }\text{ is a bijective
(hence identity) map.}
\end{equation*}%
Then for choosing $A=\left\{ 3/a,2/b,5/c,1/d\right\} $ we have 
\begin{equation*}
\left( f\left( A\right) \right) ^{c}=\left\{ 4/s,7/t,2/x,6/y,5/z\right\}
\not\subseteq \left\{ 4/s,0/t,2/x,6/y,5/z\right\} =f\left( A^{c}\right) .
\end{equation*}
\end{example}

For any new generalization one naturally expects divergences from the
previous classical theory. For if there are no divergences, generally such
generalizations do not prove fruitful. There are many statements in Theorem %
\ref{th_AlgOfFns} which are either diverging or even reversing the earlier
classical results. Statements $\left( 6b-7,9-11,13-14\right) $ are unusual
in this sense. We see these divergences as potential budding sites for new
and richer developments of the theory of multiset computation.

\section{\protect\bigskip \textbf{Interconnections }\label{sec_Interlinks}}

This section interconnects the multiset mapping presented herein with other
important notions. For easy referencing different types of mappings shall be
referred by the names of their respective authors e.g. Parikh, Nazmul,
Khomenko and Kharal mappings.

Kharal mappings bear many seminal links with other notions of multiset
processing: $\left( 1\right) $ Kharal maps possess enough modeling
capability to suitably model other important notions of mappings like
representing Parikh and Khomenko mappings. $\left( 2\right) $ Kharal maps
are generalized enough to include other mapping notions as special cases.
Specifically the works of Hickman \cite{hickman80} and Nazmul \textit{et al}%
. \cite{13namul} are special cases of Kharal maps. $\left( 3\right) $ Kharal
maps nicely interact with some of the naive measures of pattern recognition
on multisets e.g. cardinality, distance and similarity.

\subsection{Kharal Representation of Parikh and Khomenko Mappings}

Parikh mappings (vectors) express properties of words of a context free
language as numerical properties of vectors yielding some fundamental
language-theoretic consequences. Parikh mapping is used in diverse areas of
applications for example in text finger printing \cite{amir03}. Certain
shortcomings in the notion of Parikh mapping have also been pointed out in
literature \cite{Karhumaki80,Mateescu01}. For example it is noted that much
information is lost in the transition from a word to a vector. That is why,
a sharpening of the Parikh mapping, where more information is preserved than
in the original Parikh mapping, was introduced in \cite{Mateescu01}.
Different generalizations of the notion of Parikh Mapping have also been
attempted \cite{Karhumaki80}.

In the following we establish the connection between Parikh mapping and
Kharal mapping based upon multiset processing as a common denominator of the
two proposals. The approach adopted is, as usual, to show how to define one
formalism in terms of other and vice a versa.

\begin{definition}
\cite{parikh66} Let $X=\left\{ t_{1},t_{2},...,t_{k}\right\} $ be a set with
the order given by subscripts. The Parikh mapping of a multiset $%
A:X\rightarrow \mathbb{N}$ is denoted by $\psi \left( A\right) $ and is
defined as%
\begin{equation*}
\psi \left( A\right) =\left( A\left( t_{1}\right) ,A\left( t_{2}\right)
,~...~,A\left( t_{k}\right) \right)
\end{equation*}
\end{definition}

For a multiset space $X^{m}$ set $Y^{n}$ such that $n=1$ and 
\begin{equation*}
Y=\dbigcup\limits_{t_{i}\in X}\left\{ \left( 0,0,\ldots ,A\left(
t_{i}\right) ,0,\ldots ,0\right) ~|~1\leq i\leq k\right\}
\end{equation*}%
then for $f:X^{m}\rightarrow Y^{n}$ define%
\begin{equation*}
u\left( t_{i}\right) =\left( 0,0,\ldots ,A\left( t_{i}\right) ,0,\ldots
,0\right) \text{ such that }A\left( t_{i}\right) \text{ is at the }i\text{%
-th place}
\end{equation*}%
and~ $p:m^{\ast }\rightarrow n^{\ast }=\left\{ 0,1\right\} $ an identity
mapping given as%
\begin{equation*}
p\left( x\right) =\left\{ 
\begin{tabular}{ll}
$0$ & if~ $x=0$ \\ 
&  \\ 
$1$ & otherwise%
\end{tabular}%
\right.
\end{equation*}%
Then 
\begin{equation*}
\psi \left( A\right) =\sum f\left( A\right)
\end{equation*}%
where sum is the usual vector sum of members of $Y.$ Observe that $f$ is $u$%
-injective $p$-bijective mapping. Also note that $Y$ already incorporates
all possible orders on $X:$ for different orders on $X,$ only the
assignments of $u$ are to be changed. Following example illustrates the
Kharal representation of a Parikh mapping. Note that many other
possibilities may also be handled by Kharal mappings amongst which Parikh
mapping is just one.

\begin{example}
$X=\left\{ a,b,c,d,e\right\} ,$ $m=5.$ Choose $A=\left[ 3/e,4/a,3/b,0/d,1/c%
\right] \in X^{5},$ where square brackets denote that $A$ is an ordered
multiset. Then its Parikh mapping is given as $\psi \left( A\right) =\left(
3,4,3,0,1\right) .$ Now we set%
\begin{eqnarray*}
Y &=&\{\left( 0,0,0,0,0\right) ,\left( 3,0,0,0,0\right) ,\left(
0,3,0,0,0\right) ,\left( 0,0,3,0,0\right) ,\left( 0,0,0,3,0\right) ,\left(
0,0,0,0,3\right) , \\
&&\left( 4,0,0,0,0\right) ,\left( 0,4,0,0,0\right) ,\left( 0,0,4,0,0\right)
,\left( 0,0,0,4,0\right) ,\left( 0,0,0,0,4\right) , \\
&&\left( 1,0,0,0,0\right) ,\left( 0,1,0,0,0\right) ,\left( 0,0,1,0,0\right)
,\left( 0,0,0,1,0\right) ,\left( 0,0,0,0,1\right) \}
\end{eqnarray*}%
where parenthesis denote ordered pairs. $f=\left( u,p\right)
:X^{5}\rightarrow Y^{1}$ is given as follows: Order of $A$ forces following
assignments for $u:$%
\begin{eqnarray*}
u\left( a\right) &=&\left( 0,4,0,0,0\right) ,~u\left( b\right) =\left(
0,0,3,0,0\right) ,~u\left( c\right) =\left( 0,0,0,0,1\right) ,~u\left(
d\right) =\left( 0,0,0,0,0\right) ,~u\left( e\right) =\left( 3,0,0,0,0\right)
\\
p\left( 0\right) &=&0,~~p\left( 1\right) =p\left( 2\right) =p\left( 3\right)
=p\left( 4\right) =p\left( 5\right) =1
\end{eqnarray*}%
Then the calculations give%
\begin{equation*}
f\left( A\right) =\left\{ 0/\left( 0,0,0,0,0\right) ,1/\left(
0,0,0,0,1\right) ,1/\left( 0,0,3,0,0\right) ,1/\left( 0,4,0,0,0\right)
,1/\left( 3,0,0,0,0\right) \right\}
\end{equation*}%
Clearly we have 
\begin{equation*}
\psi \left( A\right) =\left( 3,4,3,0,1\right) =\left( 0,0,0,0,0\right)
+\left( 0,0,0,0,1\right) +\left( 0,0,3,0,0\right) +\left( 0,4,0,0,0\right)
+\left( 3,0,0,0,0\right) =\sum f\left( A\right)
\end{equation*}
\end{example}

Another notion of mappings is defined and used in \cite{Khomenko03} by
Khomenko, is as follows:

\begin{definition}
\cite{Khomenko03} Let $A$ be a multiset over $X$ and $h:X\rightarrow Y$ is a
mapping. Then the image $h\left( A\right) $ is defined as%
\begin{equation*}
h\left( A\right) \left( y\right) =\dsum\limits_{x\in X\wedge h\left(
x\right) =y}A\left( x\right)
\end{equation*}
\end{definition}

Replacing $\dbigvee $ with $\sum $ and choosing $p:m^{\ast }\rightarrow
n^{\ast }$ to be an identity mapping, one immediately sees above definition
as a variation of Kharal mappings.

\subsection{\textbf{Hickman Mappings}}

Before considering Hickman's mapping we have:

\begin{theorem}
\label{th_forHickmanEqul}Let $f=\left( u,p\right) :X^{m}\rightarrow Y^{n}$
be a Kharal multiset mapping, $M\in X^{m}$ and $N=f\left( M\right) $. We have

$\left( 1\right) $ If $f$ is $p$-injective then $M\left( x\right) \leq
N\left( u\left( x\right) \right) .$

$\left( 2\right) $ If $\#X=\#Y,$ and $f$ is surjective then $M\left(
x\right) \geq N\left( u\left( x\right) \right) ,$ where \# denotes the
crdinality of a set.
\end{theorem}

\begin{corollary}
\label{cor_forHickmanEql}If $f$ is injective then $M\left( x\right) \leq
N\left( u\left( x\right) \right) .$
\end{corollary}

Hickman introduced following notion of mapping for various applications:

\begin{definition}
\label{def_hickman}\cite{hickman80} Let $M,~N$ be multisets. Define a
multiset map $s:M\rightarrow N$ to be a function $Dom\left( M\right)
\rightarrow Dom\left( N\right) .$ We say that $s$ is m-injective if $s$ is
injective and $M\left( k\right) \leq N\left( s\left( k\right) \right) $ for
each $k\in M$ and that $s$ is m-surjective if $s$ is surjective and $M\left(
k\right) \geq N\left( s\left( k\right) \right) $ for each $k\in M.$ We say
that $s$ is m-bijective if $s$ is $m$-injective and $m$-surjective.
\end{definition}

One can easily note that Hickman's $s$ map is Kharal's $u$ map. Then an
injective Kharal map $f$ guarantees both conditions of Hickman's $m$%
-injective mappings. Specifically, $p$-injectivity implies $n\geq m$ (by
Proposition \ref{pr_OPrslts}$\left( 3\right) $) and OP property $p\left(
i\right) \geq p\left( i-1\right) $ assures $M\left( x\right) \leq N\left(
u\left( x\right) \right) ,$ by Corollary \ref{cor_forHickmanEql}. Also a
surjective Kharal mapping guarantees both conditions of Hickman's $m$%
-surjective maps. Surjectivity of $p$ assures $m\geq n$ and by Theorem \ref%
{th_forHickmanEqul}$\left( 2\right) $ we have $M\left( x\right) \geq N\left(
u\left( x\right) \right) .$ It is clear from Hickman's definition that $%
\#Dom\left( M\right) =\#Dom\left( N\right) $ which is also assured by
Theorem \ref{th_forHickmanEqul}$\left( 2\right) $ as $\#X=\#Y.$

Hickman's notion is general enough in the sense that it does not restrict $%
n. $ But the definition is restrictive, first in the sense, that it requires 
$\#Dom\left( M\right) =\#Dom\left( N\right) ,$ secondly, note that Kharal
map affords many other variations as well e.g. $f$ being $u$-injective $p$%
-surjective and $f$ being $u$-surjective $p$-bijective etc.

\subsection{Work of Nazmul \textit{et al}}

If we put $n=m\ $and $p:m^{\ast }\rightarrow n^{\ast }$ is an identity map
and setting $f=u$ i.e. map $f$ to be the same as $u,$ then the definition of
Kharal maps reduces to 
\begin{equation*}
f\left( A\right) \left( y\right) =\left\{ 
\begin{array}{ccc}
\underset{x\in u^{-1}\left( y\right) }{\dbigvee }A\left( x\right) & if & 
u^{-1}\left( y\right) \neq \phi \\ 
&  &  \\ 
0 &  & otherwise%
\end{array}%
\right.
\end{equation*}%
\begin{equation*}
f^{-1}\left( M\right) \left( x\right) =M\left( u\left( x\right) \right)
\end{equation*}%
which is exactly the Nazmul mapping with a slight change of notation as they
use symbol $f$ in the role of $u$ in Kharal mappings. Note that this notion
of mapping restricts the codomain multiset space to be $Y^{m}$ only, though $%
Y$ may be arbitrary.

\subsection{\textbf{Distance and Similarity Measures}}

\begin{definition}
\cite{mst01jen} The cardinality of a multiset $A\in X^{m}$ is $%
\#A=\dsum\limits_{x\in X}A\left( x\right) .$ In the sequel we shall use the
same symbol for cardinality of an ordinary set as well.
\end{definition}

\begin{definition}
A mapping $S:X^{m}\times X^{m}\rightarrow \left[ 0,1\right] $ is said to be
similarity measure if it satisfies following axioms:

$\left( s1\right) ~~0\leq S\left( A,B\right) \leq 1,$

$\left( s2\right) ~~$if $A=B,$ then $S\left( A,B\right) =1,$

$\left( s3\right) ~~S\left( A,B\right) =S\left( B,A\right) ,$

$\left( s4\right) ~~$if $A\subseteq B$ and $B\subseteq C,$ then $S\left(
A,C\right) \leq S\left( A,B\right) $ and $S\left( A,C\right) \leq S\left(
B,C\right) .$
\end{definition}

\begin{definition}
Distance and similarity between two multisets $A,B\in X^{m}$ are defined,
respectively, as:%
\begin{equation*}
d\left( A,B\right) =\sqrt{\dsum\limits_{x\in X}\left( A\left( x\right)
-B\left( x\right) \right) ^{2}}\text{ ~and ~}S\left( A,B\right) =\frac{1}{%
1+d\left( A,B\right) }\text{, where d is a metric on }X^{m}.
\end{equation*}%
It is easy to check that $d:X^{m}\times X^{m}\rightarrow \mathbb{R}^{+}\cup
\left\{ 0\right\} $ and $S:X^{m}\times X^{m}\rightarrow \left[ 0,1\right] ,$
as defined above, are respectively a metric and a similarity measure. It is
also clear that diameter of $X^{m}$ i.e. the maximum distance between any
two members of $X^{m}$ is given as 
\begin{equation*}
dia\left( X^{m}\right) =\sqrt{m}\times \#X
\end{equation*}%
where $\#X$ is the cardinality of the ordinary set $X.$
\end{definition}

Following result shows that Kharal mapping possesses some nice invariance
properties with respect to cardinality, distance and similarity:

\begin{theorem}
Let $f:X^{m}\rightarrow Y^{n}$ be a Kharal mapping and $A,B\in X^{m}$

$\left( 1\right) $If $d,S$ are a metric and similarity, respectively, on $%
X^{m}$ and $f$ is $u$-injective $p$-bijective, then we have 
\begin{eqnarray*}
\left( i\right) ~~d\left( A,B\right) &=&d\left( f\left( A\right) ,f\left(
B\right) \right) \\
\left( ii\right) ~~S\left( A,B\right) &=&S\left( f\left( A\right) ,f\left(
B\right) \right)
\end{eqnarray*}

$\left( 2\right) $ If $m>n$ and $f$ is $p$-surjective then $\#A\geq
\#f\left( A\right) .$

$\left( 3\right) $ If $f$ is injective then $\#A\leq \#f\left( A\right) $
\end{theorem}

\section{\protect\bigskip \textbf{Conclusion\label{sec_finalRems}}}

This paper has addressed the problem of defining a suitable notion of
mappings on multiset spaces. Main contribution of the paper is twofold: It
first defines a mapping on multiset spaces and presents several of its
properties and counter examples. Secondly, the new mapping has been shown to
possess many nice properties in relation to pattern recognition measures of
multisets like cardinality, distance and similarity. This mapping is further
shown to encompass Parikh and Khomenko mappings through suitable
representation schemes and Nazmul and Hickman mappings as its special cases.
The mapping rewrites multisets and thus enables one to model paradigms like $%
P$ systems, Petri Nets, Abstract Rewriting on Multisets (ARMS) and Abstract
Chemical Machines. The paper also gives several fundamental results. By
defining the notion of constant multisets, it shows that set theoretic forms
of Law of Excluded Middle and Law of Non-Contradiction do not hold for
multisets. This is the motivation to introduce $m$-coincidence to handle
disjoint multisets.\bigskip

\begin{acknowledgement}
The authors at King Saud University, extend their appreciation to the
Deanship of Scientific Research at King Saud University for funding this
work through research group no. RG-1441-439 \bigskip
\end{acknowledgement}

\end{document}